\documentclass[a4paper,10pt]{article}

\usepackage[UKenglish]{babel}
\usepackage{amsmath,amscd,amsthm,amssymb}
\usepackage{mathtools}

\theoremstyle{plain}
\newtheorem{theorem}{Theorem}[section]
\newtheorem{lemma}[theorem]{Lemma}
\newtheorem{proposition}[theorem]{Proposition}
\newtheorem{corollary}[theorem]{Corollary}

\theoremstyle{definition}

\newtheorem{example}[theorem]{Example}

\theoremstyle{remark}

\renewcommand\footnotemark{}

\DeclareMathOperator{\Irr}{Irr}
\DeclareMathOperator{\tr}{tr}

\DeclareMathOperator{\Rad}{Rad}
\DeclareMathOperator{\Mat}{Mat}
\DeclareMathOperator{\Hom}{Hom}
\DeclareMathOperator{\rk}{rk}
\DeclareMathOperator{\ord}{ord}

\DeclareMathOperator{\Ker}{Ker}
\DeclareMathOperator{\Pf}{Pf}

\begin{document}

\title{Enumerating fibres of commutator words over $p$-groups}
\author{Matthew Levy
\\\\Bielefeld University}
\date{}
\maketitle

\begin{abstract}
We enumerate the fibres of commutator word maps over $p$-groups of nilpotency class less than $p$ with exponent $p$. We also give some examples and enumerate the fibre sizes of all word maps over $p$-groups of class $2$ with exponent $p$.

\end{abstract}


\section{Introduction}

Let $G$ be a finite group, $w(x_1,...,x_n)$ a group word and $w$ the associated word map $w:G^{(n)}\rightarrow G$. For $g\in G$ denote by $N_w^G(g)$ the number of solutions to $w= g$ and by $P_w^G(g)$ the probability that a random $n$-tuple $\textbf{g}=(g_{1},...,g_{n})\in G^{(n)}$ satisfies $w(\textbf{g})=g$, i.e.\
$$
P^G_w(g) = \frac{N^G_w(g)}{|G|^n}.
$$ 
The study of $P_w^G$ has attracted a lot of attention in recent years. For example, in \cite{DPSSh}, it was shown that for $1\neq w$ and a finite simple group $G$, $P_w^G(1)\rightarrow 1$ as $|G|\rightarrow\infty$. Results in \cite{LaSh} provide sharp bounds on $P_w^G$ for general words $w$ and finite simple groups $G$. If $G$ is abelian, then the word map
$$
w:G^{(n)}\rightarrow G,
$$ 
is a homomorphism and it is clear that 
$$
N_w^G(1)=|\mbox{Ker } w|=\frac{|G|^{n}}{|\mbox{Im }w|}\geq |G|^{n-1}
$$ 
and so $P_w^G(1)\geq\frac{1}{|G|}$. It is a conjecture of Alon Amit (see \cite{Abert}) that if $G$ is a nilpotent group then $P_w^G(1)\geq\frac{1}{|G|}$. In \cite{Levy1} we prove Amit's Conjecture in the special case where the nilpotency class is $2$. Note that since the statistics $N_w^G(1)$ and $P_w^G(1)$ are multiplicative under direct products, we may reduce to the case where $G$ is a finite $p$-group. In this paper we study the fibre sizes of a particular class of words over $p$-groups. In Section \ref{class2} we show that this is enough to determine the fibre sizes of all word maps over $p$-groups of class $2$ with exponent $p$.

For $t\in\mathbb{N}$, let $c_t$ denote the word map given by $c_t = [x_1,y_1]...[x_t,y_t]$ and, for $g\in G$, let $N_t^G(g)$ denote $N_{c_t}^G(g)$. Similarly let $P_t^G$ denote the \textit{$c_t$-distribution} on $G$, i.e.\
$$
P_t^G(g) = \frac{N_t^G(g)}{|G|^{2t}},
$$
and let $U^G$ be the uniform distribution on $G$ (i.e. $U^G(g) = 1/|G|$). By a classical result of
Frobenius from 1896 (see, for example, \cite{Isaacs}) we have
\begin{equation}\label{equation1}
N_t^G(g) = |G|^{2t-1}\sum_{\chi\in\Irr(G)}\frac{\chi(g)}{\chi(1)^t},
\end{equation}
where $1$ is the identity element of $G$ and we sum over the irreducible complex characters of $G$. In the main result of this paper, Theorem \ref{main}, we give explicit formulae to compute the numbers $N^G_t(g)$, for finite $p$-groups $G$ of nilpotency class less than $p$ with exponent $p$, in terms of the number of rational points of certain algebraic varieties. 

In \cite{GarionShalev} Garion \& Shalev prove the following.

\begin{proposition}[Proposition 1.1 \cite{GarionShalev}]
Let $G$ be a finite group. Then
$$
||P_1^G-U^G||_1 \leq  \left(\sum_{\chi\in\Irr(G), \chi\neq\text{Id}}\chi(1)^{-2}\right)^{1/2},
$$
where $||P_1^G-U^G||_1=\sum_{g\in G}|P_1^G(g)-U^G(g)|$.
\end{proposition}

As the authors remark, this bound has no content when the sum on the right hand side is greater than or equal to $1$. Since the non-trivial linear characters of $G$ contribute $|G/G'|^{-1}$ to the sum the result can only be useful for perfect groups. Since the maps $c_t$ take values in $G'$ it is not hard to adapt their proof and deduce the following.

\begin{proposition}\label{propdist}
Let $G$ be a finite group. Then
$$
||P_t^{G'}-U^{G'}||_1 \leq \left(\frac{|G'|}{|G|}\sum_{\chi\in\Irr(G), \chi(1)\neq 1}\chi(1)^{-2t}\right)^{1/2}.
$$
\end{proposition}

If the right hand side in Proposition \ref{propdist} is close to zero, then the fibres of the maps $c_t$ are roughly the same size and the values are uniformly distributed. The proof follows easily from the following lemmas.
\begin{lemma}\label{prevlemma}
Let $G$ be a finite group. Then
$$
\sum_{g\in G'} P^G_t(g)^2= \frac{1}{|G'|}+\frac{1}{|G|}\sum_{\chi\in\Irr(G),\chi(1)\neq 1}\chi(1)^{-2t}.
$$
\end{lemma}
\begin{proof}
The proof follows from Lemma 2.1 in \cite{GarionShalev} and noting that $P^G_t(g) = 0$ for $g\not\in G'$.
\end{proof}
\begin{lemma}
Let $G$ be a finite group. Then
$$
\sum_{g\in G'}\left(P^G_t(g)-\frac{1}{|G'|}\right)^2=\frac{1}{|G|}\sum_{\chi\in\Irr(G),\chi(1)\neq 1}\chi(1)^{-2t}.
$$
\end{lemma}
\begin{proof}
By Lemma \ref{prevlemma},
\begin{eqnarray*}
\sum_{g\in G'}\left(P^G_t(g)-\frac{1}{|G'|}\right)^2 &=& \sum_{g\in G'}P^G_t(g)^2-\frac{2}{|G'|}\sum_{g\in G'}P^G_t(g)+\frac{1}{|G'|}\\ &=&\frac{1}{|G|}\sum_{\chi\in\Irr(G),\chi(1)\neq 1}\chi(1)^{-2t}
\end{eqnarray*}
since $\sum_{g\in G'}P^G_t(g) = 1$.
\end{proof}

\begin{proof}[Proof of Proposition \ref{propdist}]
This follows from the Cauchy-Schwarz inequality,
$$
(||P_t^{G'} - U^{G'}||_1)^2 = \left(\sum_{g\in G'}(P_t(g)-\frac{1}{|G'|}\right)^2\leq|G'|\sum_{g\in G'}\left(P_t(g)-\frac{1}{|G'|}\right)^2
$$
and the previous lemma.
\end{proof}

In Section \ref{mainsec} we will develop formulae for the fibre sizes of the word maps $c_t$ over $p$-groups of nilpotency class less than $p$ with exponent $p$. Moreover, these results extend, more generally, to $p$-groups obtained by `base extension'. In Section \ref{class2} we will determine the fibre sizes of all word maps over $p$-groups of class $2$ with exponent $p$. We will then give some examples in Section \ref{examples}. 

\section{Enumerating sizes of fibres}\label{mainsec}

Let $G$ be a finite $p$-group and, for each $i\in\mathbb{N}$, write $\Irr^i(G) = \{$irreducible complex characters of $G$ of degree $p^i\}$ and $\Irr(G)$ for the set of all irreducible complex characters. 
For a complex variable $s$ and $g\in G$ write
\begin{equation}\label{twistedzeta}
\zeta^i_G(s,g) = \sum_{\chi\in\Irr^i(G)}\frac{\chi(g)}{\chi(1)^s}.
\end{equation}
We will also write
\begin{equation}\label{twistedzetamain}
\zeta_G(s,g) = \sum_{\chi\in\Irr(G)}\frac{\chi(g)}{\chi(1)^s}=\sum_{i\in\mathbb{N}}\zeta^i_G(s,g).
\end{equation}
It is clear from equations (\ref{equation1}) and (\ref{twistedzetamain}) that
$$
N^G_t(g) = |G|^{2t-1}\zeta_G(t,g)
$$
and that
\begin{equation}\label{equationprobzeta}
P^G_t(g) = \frac{1}{|G|}\zeta_G(t,g).
\end{equation}
In this section we prove Theorem \ref{main} which allows us to compute expressions like equations (\ref{twistedzetamain}) and (\ref{equationprobzeta}) for finite $p$-groups of nilpotency class less than $p$ with exponent $p$.

Note that when $s=1$ and $g=1$, the identity of $G$, the sum $\zeta_G(1,1)$ is simply the class number $k(G)$ of $G$ and that if $g$ is not in the derived group of $G$ then $\zeta_G(s,g)=0$ since $N_s^G(g)=0$, see equation (\ref{equation1}). An analogue of equation (\ref{twistedzetamain}), a \textit{twisted} zeta function, is studied by Jaikin-Zapirain c.f.\ \cite[Theorem 1.2]{Jaikinzeta} where he shows that it is a rational function in $p^{-s}$ for any $g\in G$ where $G$ is a FAb uniform pro-$p$ group. 

We also note that the twisted zeta function in equation (\ref{twistedzetamain}) is multiplicative under direct products in the following sense: for $g_1\in G_1$ and $g_2\in G_2$ where $G_1$ and $G_2$ are groups we have
$$
\zeta_{G_1\times G_2}(s,g_1 g_2) = \zeta_{G_1}(s,g_1).\zeta_{G_2}(s,g_2).
$$

In \cite{VollOBrien} O'Brien \& Voll use the Kirillov orbit method to enumerate the irreducible complex characters, of each degree, of finite $p$-groups of nilpotency class less than $p$. If the group $G$ is of exponent $p$, then the number of characters, of each degree, of $G$ can be described in terms of the number of rational points of certain algebraic varieties. 
We can analogously compute the sums $\zeta_G^i(s,g)$ by counting the numbers of rational points of the algebraic varieties considered by O'Brien \& Voll that intersect a hyperplane characterized by $g$.

We now fix a $p$-group $G$ of nilpotency class less than $p$ with exponent $p$ and an element $g\in G'$.

The Lazard correspondence establishes an order-preserving bijection between finite $p$-groups of nilpotency class $c<p$ and finite nilpotent Lie rings of $p$-power order and class $c<p$; cf.\ \cite[Example 10.24]{Khukhro}. Let $c<p$ be the nilpotency class of $G$. Let $\mathfrak{g}=\log(G)$ be the finite Lie ring associated to $G$ by the Lazard correspondence. The Kirillov orbit method gives a correspondence between characters of $G$ and orbits in $\hat{\mathfrak{g}}:=\Hom_{\mathbb{Z}}(\mathfrak{g},\mathbb{C}^*)$, the Pontryagin dual of $\mathfrak{g}$, under the co-adjoint action of $G$ on $\hat{\mathfrak{g}}$; cf.\ \cite[Theorem 2.6]{BoySab} or \cite[Theorem 4.4]{Jon}. Under this correspondence each orbit $\Omega$ of size, say, $p^{2i}$ gives rise to a character $\chi_{\Omega}$ of degree $p^i$ and all characters are of this form. We have
$$
\zeta^i_G(s,g)=\sum_{\Omega\subseteq\hat{\mathfrak{g}}, |\Omega|=p^{2i}}\frac{\chi_{\Omega}(g)}{\chi_{\Omega}(1)^s},\\
$$
where we sum over orbits $\Omega$ of $\hat{\mathfrak{g}}$ and $\chi_{\Omega}$ is the character of $G$ that corresponds to the orbit $\Omega$. For each $\omega\in\Omega$ we denote $\chi_{\omega} = \chi_{\Omega}$. 

For each $\omega\in\hat{\mathfrak{g}}$ we write $B_{\omega}$ for the bi-additive, skew-symmetric form $B_{\omega}:\mathfrak{g}\times\mathfrak{g}\rightarrow\mathbb{C}^*$, $(u,v)\mapsto\omega([u,v])$ and $\Rad(B_{\omega})$ for the radical of $B_{\omega}$. Note that the centre $\mathfrak{z}$ of $\mathfrak{g}$ is contained in $\Rad(B_{\omega})$ and the form $B_{\omega}$ only depends on the restriction of $\omega$ to $\mathfrak{g}'$. From \cite{VollOBrien} we have, for each $i$,
\begin{eqnarray*}
\zeta^i_G(s,g)&=&p^{-2i}\sum_{\omega\in\hat{\mathfrak{g}}, |\mathfrak{g}:\Rad(B_{\omega})|=p^{2i}}\frac{\chi_{\omega}(g)}{\chi_{\omega}(1)^s}\\
&=&|\mathfrak{g}/\mathfrak{g'}|p^{-2i}\sum_{\omega\in\widehat{\mathfrak{g'}}, |\mathfrak{g}:\Rad(B_{\omega})|=p^{2i}}\frac{\chi_{\omega}(g)}{\chi_{\omega}(1)^s}\\
&=&|\mathfrak{g}/\mathfrak{g'}|p^{-2i}\sum_{\omega\in\widehat{\mathfrak{g'}}, |\Rad(B_{\omega}):\mathfrak{z}|=p^{-2i}|\mathfrak{g}/\mathfrak{z}|}\frac{\chi_{\omega}(g)}{\chi_{\omega}(1)^s}.\\
\end{eqnarray*}
By the Kirillov orbit method we have, for all $\omega\in\hat{\mathfrak{g}}$ such that $|\mathfrak{g}:\Rad(B_{\omega})|=p^{2i}$,
$$
\chi_{\omega}(g) = p^{-i}\sum_{\upsilon\in\Omega_{\omega}}\upsilon(g),
$$
where $\Omega_{\omega}$ is the orbit containing $\omega$ and we identify $g$ with an element of $\mathfrak{g}=\log(G)$. Hence
$$
\zeta^i_G(s,g)=|\mathfrak{g}/\mathfrak{g'}|p^{-i(3+s)}\sum_{\omega\in\widehat{\mathfrak{g'}}, |\Rad(B_{\omega}):\mathfrak{z}|=p^{-2i}|\mathfrak{g}/\mathfrak{z}|}\sum_{\upsilon\in\Omega_{\omega}}\upsilon(g).
$$
 
Theorem B in \cite{VollOBrien} gives a geometric characterization of this sum in terms of the numbers of certain rational points of rank varieties of matrices of linear forms. Assume that $\mathfrak{o}$ is a compact discrete valuation ring of characteristic zero with maximal ideal $\mathfrak{p}$ and residue field $\bold{k} = \mathfrak{o}/\mathfrak{p}$ of characteristic $p$. Suppose that $\mathfrak{g}$ is a finite, nilpotent $\mathfrak{o}$-Lie algebra of class $c<p$ and that both $\mathfrak{g}/\mathfrak{z}$ and $\mathfrak{g}'$ are annihilated by $\mathfrak{p}$. Set $a:=\rk_{\bold{k}}(\mathfrak{g}/\mathfrak{z})$ and $b:=\rk_{\bold{k}}(\mathfrak{g'})$ and fix bases $\bold{e}=\{e_1,...,e_a\}$ and $\bold{f}=\{f_1,...,f_b\}$ for $\mathfrak{g}/\mathfrak{z}$ and $\mathfrak{g'}$ respectively. Choose structure constants $\lambda_{ij}^k\in\bold{k}$ such that
$$
[e_i,e_j] = \sum_{k=1}^{b}\lambda_{ij}^kf_k
$$
and $\lambda_{ij}^k = -\lambda_{ji}^k$ for all $i,j\in\{1,...,a\}$ and $k\in\{1,...,b\}$. Let $\bold{Y} = (Y_1,...,Y_b)$ be independent variables and define the `commutator matrix' (with respect to $\bold{e}$ and $\bold{f}$) $B(\bold{Y})\in\Mat_a(\bold{k}[\bold{Y}])$ given by
$$
B(\bold{Y})_{ij}:=\sum_{k=1}^{b}\lambda_{ij}^{k}Y_k
$$
for all $i,j\in\{1,...,a\}$. For $\bold{y}=(y_1,...,y_b)\in\bold{K}^b$, where $\bold{K}$ is any finite extension of $\bold{k}$, write $B(\bold{y})\in\Mat_a(\bold{K})$ for the matrix obtained by evaluating the variables $Y_i$ at $y_i$. Note that the matrices $B(\bold{y})$ are skew-symmetric, have even rank
and that $\det(B(\bold{Y}))$ is a square in $\bold{k}[\bold{Y}]$, whose square root $\Pf(B(\bold{Y})):=\sqrt{\det(B(\bold{Y}))}$ is the \textit{Pfaffian} of $B(\bold{Y})$. If $a$ is odd then $\Pf(B(\bold{Y}))=0$. 

Fix a non-trivial additive character $\phi:\bold{k}\rightarrow\mathbb{C}^*$. For $a\in\bold{k}$ define $\phi_a(x) = \phi(ax)$. The map $a\mapsto\phi_a$ is an isomorphism between $\bold{k}$ and its Pontryagin dual $\widehat{\bold{k}}$. 
We have an isomorphism between $\mathfrak{g}'$ and its dual $\widehat{\mathfrak{g}'}$ and also a canonical isomorphism between $\mathfrak{g}'$ and its linear dual $\Hom_{\mathbb{F}_q}(\mathfrak{g}',\bold{k})$. We now fix an isomorphsim $\psi_1:\widehat{\mathfrak{g}'}\rightarrow\Hom_{\bold{k}}(\mathfrak{g}',\bold{k})$. The dual $\bold{k}$-basis $f^{\vee} = (f_k^{\vee})$ for $\Hom_{\bold{k}}(\mathfrak{g}',\bold{k})$ gives a coordinate system
\begin{eqnarray*}
\psi_2:\Hom_{\bold{k}}(\mathfrak{g}',\bold{k})&\rightarrow&\bold{k}^b;\\
y=\sum_{k=1}^{b}y_kf_f^{\vee}&\mapsto&\bold{y}=(y_1,...,y_b).
\end{eqnarray*}
Set $\psi:=\psi_2\circ\psi_1:\widehat{\mathfrak{g}'}\rightarrow\bold{k}^b$, an isomorphism. Under this isomorphism we may identify elements $\bold{y}\in\bold{k}^b$ with elements $\omega_{\bold{y}}\in\widehat{\mathfrak{g}'}$. 


Assume that $\mathfrak{O}$ is an unramified extension of $\mathfrak{o}$, with maximal ideal $\mathfrak{P}$. Write $\mathfrak{g}(\mathfrak{O})$ for $\mathfrak{g}\otimes_{\mathfrak{o}}\mathfrak{O}$ and $\mathfrak{z}(\mathfrak{O})$ for $\mathfrak{z}\otimes_{\mathfrak{o}}\mathfrak{O}$. We identify the residue field $\mathfrak{O}/\mathfrak{P}$, a finite extension of $\bold{k}$, with $\mathbb{F}_q$. The derived $\mathfrak{O}$-Lie algebra $\mathfrak{g}(\mathfrak{O})'$ as well as the quotient $\mathfrak{g}(\mathfrak{O})/\mathfrak{z}(\mathfrak{O})$ are annihilated by $\mathfrak{P}$. We denote the corresponding bases, $\bold{e}\otimes_{\bold{k}}1$ for $\mathfrak{g}(\mathfrak{O})/\mathfrak{z}(\mathfrak{O})$ and $\bold{f}\otimes_{\bold{k}}1$ for $\mathfrak{g}(\mathfrak{O})'$, obtained by this base extension by $\bold{e}$ and $\bold{f}$ respectively. We consider $\bold{e}$ and $\bold{f}$ as $\mathbb{F}_q$ bases for the respective $\mathbb{F}_q$-vector spaces of dimensions $a$ and $b$. The group $G=G(\mathfrak{O}):=\exp(\mathfrak{g}(\mathfrak{O}))$ has the property that both $G(\mathfrak{O})'$ and $G(\mathfrak{O})/Z(G(\mathfrak{O}))$ have exponent $p$. It follows from \cite[Theorem B]{VollOBrien} that
$$
\zeta^i_G(s,g)=|\mathfrak{g}/\mathfrak{g'}|p^{-i(3+s)}\sum_{\bold{y}\in\mathbb{F}_q^b,  \rk(B(\bold{y}))=2i}\sum_{\bold{v}\in\Omega_{\omega_\bold{y}}}\bold{v}(g).
$$
Since $|\Omega_{\omega_\bold{y}}| = p^{2i}$ we have
$$
\zeta^i_G(s,g)=|\mathfrak{g}/\mathfrak{g'}|p^{-i(1+s)}\sum_{\bold{y}\in\mathbb{F}_q^b,  \rk(B(\bold{y}))=2i}\omega_{\bold{y}}(g).
$$
Let 
\begin{equation}\label{eqnK}
K^i_G(g)=\#\{\bold{y}\in\mathbb{F}_q^b:  \rk(B(\bold{y}))=2i, g\in\Ker(\omega_{\bold{y}})\}
\end{equation} 
and 
\begin{equation}\label{eqnV}
V^i_G(g)=\#\{\bold{y}\in\mathbb{F}_q^b:  \rk(B(\bold{y}))=2i, \ord(\omega_{\bold{y}}(g))\neq 1\}.
\end{equation}
Note that $\sum_i K^i_G(1) = q^b$ and $V_G^i(1)=0$ for all $i$ whilst for $1\neq g\in G$ we have $\sum_i K^i_G(g)=q^{b-1}$. 

We write $K_G(g)$ and $V_G(g)$ for the vectors $(K^i_G(g))_i$ and $(V^i_G(g))_i$ respectively. The numbers $\omega_{\bold{y}}(g)$ are $p$-th roots of unity. Since the sum of the $q-1$ Galois conjugates of non-trivial $q$-th roots of unity is $-1$ we have the following:
\begin{theorem}\label{main}
Let $\mathfrak{o}$ be a compact discrete valuation ring of characteristic zero with residue field $\bold{k}$ of characteristic $p$ and let $\mathfrak{g}$ be a finite, nilpotent $\mathfrak{o}$-Lie algebra of class $c<p$. Assume that $\mathfrak{g}'\cong\bold{k}^b$ and $\mathfrak{g}/\mathfrak{z}\cong\bold{k}^a$ as $\bold{k}$-vector spaces. Let $\mathfrak{O}$ be a finite, unramified extension of $\mathfrak{o}$, with residue field isomorphic to $\mathbb{F}_q$. Write $G:=G(\mathfrak{O})$ for the group associated with $\mathfrak{g}(\mathfrak{O})$ under the Lazard correspondence. Then, for each $i$,
$$
\zeta^i_G(s,g)=|G/G'|p^{-i(1+s)}(K^i_G(g)-\frac{1}{q-1}V^i_G(g)),
$$
where $K^i_G$ and $V^i_G$ are defined in equations (\ref{eqnK}) and (\ref{eqnV}).
\end{theorem}
\noindent This follows immediately from the discussion above.



\section{Nilpotent class $2$ groups}\label{class2}

Let $G$ be a finite group, $w(x_1,...,x_n)$ a group word and $w$ the associated word map. Recall that $N_w^G(g)$ is the number of solutions to $w= g$ 
and that $P_w^G(g)$ is the probability that a random $n$-tuple $\textbf{g}=(g_{1},...,g_{n})\in G^{(n)}$ satisfies $w(\textbf{g})=g$. 
The following corollary follows immediately from the results in \cite{Levy1}:

\begin{corollary}[\cite{Levy1}]
Let $G$ be a finite $p$-group of nilpotency class $2$ with exponent $p$ and let $w$ be a group word. Then there exists a word $v$ of the following form:
\begin{itemize}
\item[i)] $v = x$; or
\item[ii)] $v = c_t$ for some $t$,
\end{itemize}
such that $P_w^G(g) = P_v^G(g)$ for all $g\in G$.
\end{corollary}

By Theorem \ref{main} we have determined the fibre sizes of all word maps over all $p$-groups of nilpotency class $2$ with exponent $p$:

\begin{corollary}
Let $p$ be an odd prime and let $G$ be a finite $p$-group of nilpotency class $2$ with exponent $p$ and let $w$ be a group word. Then either
\begin{itemize}
\item[i)] $P_w^G(g) = \frac{1}{|G|}$ for all $g\in G$; or
\item[ii)] $P_w^G(g) = \frac{1}{|G|}\zeta_G(t,g)$ for some $t$ depending on $w$.
\end{itemize}
\end{corollary}

\section{Examples}\label{examples}

We compute the sums $\zeta^i_G(s,g)$ for various relatively free $p$-groups with exponent $p$. Note that for $i=0$, the only vector $\bold{y}$ giving rise to a matrix $B(\bold{y})$ such that $\rk(B(\bold{y}))=0$ (see equations (\ref{eqnK}) and (\ref{eqnV})) is $\bold{y} = \bold{0}$. Suppose that $i\neq 0$. Since the rank of the matrix $B(\bold{y})$, for given $\bold{y}\in\mathbb{F}_q^b$, is invariant under scalar multiplication the rank of the matrix $B(\tilde{\bold{y}})$ with $\tilde{\bold{y}}=(\tilde{y_1}:...:\tilde{y_b})\in\mathbb{P}^{b-1}(\mathbb{F}_q)$ is well defined. Under the Lazard correspondence we may identify $1\neq g\in G'$ with an element $\bold{g}\in\mathfrak{g}'\cong\mathbb{F}_q^{b}$, the associated Lie algebra, and similarly identify $\bold{g}$ with $\tilde{\bold{g}}\in\mathbb{P}^{b-1}(\mathbb{F}_q)$. Since the matrix $B$ is skew-symmetric its determinant $\det(B)$ is a square whose square root $\Pf(B):=\sqrt{\det(B)}$ is the \textit{Pfaffian} of $B$. If $a$ is odd, then $\Pf(B) = 0$. Assume that $\Pf(B) \neq 0$. Then $\Pf(B)$ defines a hypersurface in $\mathbb{P}^{b-1}$ and the $\mathbb{F}_q$-rational points $\tilde{\bold{y}}$ of this hypersurface correspond to matrices $B(\tilde{\bold{y}})$ of a certain non-maximal rank. The $\mathbb{F}_q$-rational points which do not lie on the hypersurface will be of maximal rank $       a$. We refer to points $\tilde{\bold{y}}\in\mathbb{P}^{b-1}(\mathbb{F}_q)$ as being `of rank 2i' for some $i$ if the associated matrix $B(\tilde{\bold{y}})$ is of rank $2i$. The condition $g\in\Ker(\omega_{\bold{y}})$ in the definition of $K^i_G(g)$ (see equation (\ref{eqnK})) defines a hyperplane $H_g$ given by the dot product $\tilde{\bold{g}}.\tilde{\bold{y}}=0$ in $\mathbb{P}^{b-1}(\mathbb{F}_q)$. We may thus talk about the `hyperplane defined by $g$' as $H_g$. The numbers $K^i_G(g)$ are simply the number of $\mathbb{F}_q$-rational points of this hyperplane which intersect the hypersurface defined by $\Pf(B)$ in $\mathbb{P}^{b-1}(\mathbb{F}_q)$ giving rise to a point `of rank $2i$'.

We will now proceed with some examples. In each case it turns out that the matrices $B(\bold{Y})$ have the form
$$
B(\bold{Y}) =
\begin{pmatrix}
0&U(\bold{Y})\\
-U(\bold{Y})^{\tr}& 0
\end{pmatrix}
$$
where $U(\bold{Y})$ is a matrix.  

\begin{example}[Heisenberg group, H($\mathbb{F}_q$)]
Suppose that $G$ is the Heisenberg group H($\mathbb{F}_q$). The matrix $U(\bold{Y})$, where $\bold{Y}=(Y_1)$ has a single variable, is simply the $1\times 1$ matrix with entry $Y_1$. There are two cases, $i=0$ and $i=1$, for the rank of the matrix $U(\bold{Y})$. The number of vectors $\bold{y}$ in each case is $1$ and $q-1$ respectively. When $g=1$ is the identity, we have $K^0_G(1)=1$ and $V^0_G(1) = 0$ so that $\zeta^0_G(s,1)=q^2$. We also have $K^1_G(1)=q-1$ and $V^1_G(1) = 0$ so that $\zeta^1_G(s,1)=q^{1-s}(q-1)$. Together, this gives us 
$$
\zeta_G(s,1)=q^2+q^{-s+1}(q-1)
$$ 
and when $s=1$ this is simply $k(G)$, the class number. 

Suppose now that $1\neq g\in G'$. This occurs with multiplicity $(q-1)$. It is not hard to see that $\zeta^0_G(s,g)=q^2$ since $K^0_G(g)=1$ and $V^0_G(g) = 0$. Also, $K^1_G(g)=0$ and $V^1_G(g) = q-1$ so that 
$$
\zeta_G(s,g)=q^2-q^{-s+1}.
$$

\end{example}

\begin{example}[A quadric surface in $\mathbb{P}^2(\mathbb{F}_q)$]
Let $\mathfrak{g}$ be the $7$-dimensional nilpotent $\mathbb{F}_q$-Lie algebra of class $2$ with $\mathbb{F}_q$-basis $(x_1,...,x_4,y_1,y_2,y_3)$ subject to the relations $[x_1,x_3]=y_1$, $[x_1,x_4]=y_2$, $[x_2,x_3]=y_3$, $[x_2,x_4]=y_1$. With respect to this basis we have
$$
U(\bold{Y}) =
\begin{pmatrix}
Y_1&Y_2\\
Y_3& Y_1
\end{pmatrix}.
$$
The determinant of this matrix is defines the quadric surface $Y_1^2-Y_2Y_3$ in $\mathbb{P}^2(\mathbb{F}_q)$. There are three cases, $i=0,1$ and $2$, for the rank of the matrix $U(\bold{Y})$. The number of such vectors $\bold{y}$ in each case is $1$, $(q+1)(q-1)$ and $q^2(q-1)$ respectively. As usual, the $i=0$ case corresponds to when $\bold{y}$ is zero, $i=1$ when $\tilde{\bold{y}}$ lies on the curve and $i=2$ otherwise. When $g$ is the identity we can compute 
$$
K_G(1)=(1,(q+1)(q-1),q^2(q-1))
$$ 
and
$V_G(1)=(0,0,0)$ so that $k(G)=q^4+q^2(q+1)(q-1)+q^2(q-1)$.

Now suppose that $g\neq 1$. We have two cases. The first case is when $H_g$ corresponds to a tangent of the surface $Y_1^2=Y_2Y_3$. This occurs with multiplicity $(q+1)(q-1)$. In this case 
$$
K_G(g)=(1,(q-1),q(q-1))
$$ 
and 
$$
V_G(g)=(0,q(q-1),(q^2-q)(q-1)).
$$

The second case is when $H_g$ corresponds to a line which intersects the surface $Y_1^2=Y_2Y_3$ at two distinct points. This occurs with multiplicity $q^2(q-1)$. In this case 
$$
K_G(g)=(1,2(q-1),(q-1)(q-1))
$$ 
and 
$$
V_G(g)=(0,(q-1)(q-1),(q^2-q+1)(q-1)).
$$
\end{example}

\begin{example}[A quadric surface in $\mathbb{P}^3(\mathbb{F}_q)$]
Let $\mathfrak{g}$ be the $8$-dimensional nilpotent $\mathbb{F}_q$-Lie algebra of class $2$ with $\mathbb{F}_q$-basis $(x_1,...,x_4,y_1,...,y_4)$ subject to the relations $[x_1,x_3]=y_1$, $[x_1,x_4]=y_2$, $[x_2,x_3]=y_3$, $[x_2,x_4]=y_4$. With respect to this basis we have
$$
U(\bold{Y}) =
\begin{pmatrix}
Y_1&Y_2\\
Y_3& Y_4
\end{pmatrix}.
$$
The determinant of this matrix defines the quadric surface $Y_1Y_4-Y_2Y_3$ in $\mathbb{P}^3(\mathbb{F}_q)$. There are three cases, $i=0,1$ and $2$, for the rank of the matrix $U(\bold{Y})$. The number of vectors $\bold{y}$ in each case is $1$, $(q+1)^2(q-1)$ and $(q^3-q)(q-1)$ respectively. As usual, the $i=0$ case corresponds to when $\bold{y}$ is zero, $i=1$ when $\tilde{\bold{y}}$ lies on the curve and $i=2$ otherwise. When $g$ is the identity we compute 
$$
K_G(1)=(1,(q+1)^2(q-1),q^4-1-(q+1)^2(q-1))
$$ 
and 
$V_G(1)=(0,0,0)$ so that $k(G)=q^4+q^2(q+1)^2(q-1)+q^4-1-(q+1)^2(q-1)$. 

Now suppose that $g\neq 1$. We have two cases. The first case is when $g$ corresponds to a tangent of the surface $Y_1Y_4=Y_2Y_3$. This occurs with multiplicity $(q+1)^2(q-1)$. In this case 
$$
K_G(g)=(1,(2q+1)(q-1),(q^2-q)(q-1))
$$ 
and 
$$
V_G(g)=(0,q^2(q-1),(q^3-q^2)(q-1)).
$$

The second case is when $g$ corresponds to a plane which intersects the surface $Y_1Y_4=Y_2Y_3$ and is non-tangent. This occurs with multiplicity $(q^3-q)(q-1)$. In this case 
$$
K_G(g)=(1,(q+1)(q-1),q^2(q-1))
$$ 
and 
$$
V_G(g)=(0,q(q+1)(q-1),(q^3-q^2-q)(q-1)).
$$
\end{example}

\noindent The following examples were studied by Boston \& Isaacs \cite{BostonIsaacs}.

\begin{example}[Elliptic curves in $\mathbb{P}^2(\mathbb{F}_p)$]
Let $p$ be a prime and $\alpha\in\mathbb{F}_p^*$. Let $\mathfrak{g}_{\alpha}$ be the $9$-dimensional nilpotent $\mathbb{F}_p$-Lie algebra of class $2$ with $\mathbb{F}_p$-basis $(x_1,...,x_6,y_1,y_2,y_3)$ subject to the relations $[x_1,x_4]=y_1$, $[x_1,x_5]=y_2$, $[x_1,x_6]=\alpha y_3$, $[x_2,x_4]=y_3$, $[x_2,x_5]=y_1$, $[x_2,x_6]=y_2$, $[x_3,x_4]=y_3$, $[x_3,x_6]=y_1$. With respect to this basis we have
$$
U(\bold{Y}) =
\begin{pmatrix}
Y_1&Y_2&\alpha Y_3\\
Y_3& Y_1&Y_2\\
Y_3&0&Y_1
\end{pmatrix}.
$$
The determinant of this matrix defines an elliptic curve $E_{\alpha}$ in $\mathbb{P}^2(\mathbb{F}_p)$ and let $n_{\alpha}$ denote the number of $\mathbb{F}_p$-rational points of the curve $E_{\alpha}$. There are three cases, $i=0,1$ and $2$, for the rank of the matrix $U(\bold{Y})$. The number of such vectors $\bold{y}$ in each case is $1$, $n_{\alpha}(p-1)$ and $(p^2+p+1-n_{\alpha})(p-1)$ respectively. As usual, the $i=0$ case corresponds to when $\bold{y}$ is zero, $i=1$ when $\tilde{\bold{y}}$ lies on the curve and $i=2$ otherwise. When $g$ is the identity we can compute 
$$
K_G(1)=(1,n_{\alpha}(p-1),(p^2+p+1-n_{\alpha})(p-1))
$$ 
and $V_G(1)=(0,0,0)$ so that $k(G)=p^6+p^2n_{\alpha}(p-1)+(p^2+p+1-n_{\alpha})(p-1)$. 

Now suppose that $g\neq 1$. Let $k_{\alpha}$ denote the number of inflection points of $E_{\alpha}$. We have four cases depending on how many times the line $H_g$ corresponding to $g$ intersects the elliptic curve $E_{\alpha}$ at rational points, this can be zero, once, twice or three times occurring with multiplicities $p^2+p+1-n_{\alpha}(p+1)+\frac{2}{3}{n_{\alpha} \choose 2}+\frac{n_{\alpha}-k}{3}$, $n_{\alpha}(p+1)-(n_{\alpha}-k)-{n_{\alpha} \choose 2}$, $(n_{\alpha}-k)(p-1)$ and ${n_{\alpha} \choose 2}$ respectively. 

Suppose that $1\neq g$ defines a line which intersects the curve at $m$ rational points where $m=0,1,2,3$. We have 
$$
K_G(g)=(1,m(p-1), (p+1-m)(p-1))
$$ 
and 
$$
V_G(g)=(0,(n_{\alpha}-m)(p-1),(p^2+m-n_{\alpha})(p-1)).
$$
\end{example}

%
%
%
%

\section{Acknowledgements}
I would like to thank Christopher Voll for his support, guidance and insight throughout this project and for numerous helpful discussions.

\bibliographystyle{plain}
\bibliography{refs}

\begin{thebibliography}{10}

\bibitem{Abert}
M.~Ab{\'e}rt.
\newblock On the probability of satisfying a word in a group.
\newblock {\em J. Group Theory}, 9(5):685--694, 2006.

\bibitem{BostonIsaacs}
N.~Boston and I.~M. Isaacs.
\newblock Class numbers of $p$-groups of a given order.
\newblock {\em J. Algebra}, 279(2):810--819, 2004.

\bibitem{BoySab}
M.~Boyarchenko and M.~Sabitova.
\newblock The orbit method for profinite groups and a $p$-adic analogue of
  brown's theorem.
\newblock {\em Israel J. Math}, 165:67--91, 2008.

\bibitem{DPSSh}
J.~D. Dixon, L.~Pyber, \'{A}. Seress, and A.~Shalev.
\newblock Residual properties of free groups and probabilistic methods.
\newblock {\em J. reine angew. Math. (Crelle's)}, 556:159--172, 2003.

\bibitem{GarionShalev}
S.~Garion and A.~Shalev.
\newblock Commutator maps, measure preservation and {T}-systems.
\newblock {\em Trans. Amer. Math. Soc.}, 361(9):4631--4351, 2009.

\bibitem{Jon}
J.~Gonz\'{a}lez-S\'{a}nchez.
\newblock Kirillov's orbit method for $p$-groups and pro-$p$ groups.
\newblock {\em Comm. Algebra}, 37(12):4476--4488, 2009.

\bibitem{Isaacs}
I.~M. Isaacs.
\newblock {\em Character Theory of Finite Groups}, volume 359.
\newblock American Math. Soc., 1976.

\bibitem{Jaikinzeta}
A.~Jaikin-Zapirain.
\newblock Zeta function of representations of compact $p$-adic analytic groups.
\newblock {\em J. Amer. Math. Soc.}, (19):91--118, 2006.

\bibitem{Khukhro}
E.~I. Khukrho.
\newblock {\em $p$-automorphisms of finite $p$-groups}, volume 246 of {\em
  London Mathematical Society Lecture Note Series}.
\newblock Cambridge University Press, Cambridge, 1198.

\bibitem{LaSh}
M.~Larsen and A.~Shalev.
\newblock Fibres of word maps and some applications.
\newblock {\em J. Algebra}, 354:36--48, 2012.

\bibitem{Levy1}
M.~Levy.
\newblock On the probability of solving a word in nilpotent groups of class
  $2$.
\newblock (ar$\chi$iv 1101.4286v1).

\bibitem{VollOBrien}
E.~A. O'Brien and C.~Voll.
\newblock Enumerating classes and characters of $p$-groups.
\newblock {\em Trans. Amer. Math. Soc.}, 367:7775--7796, 2015.

\end{thebibliography}

\end{document}